\documentclass[reqno,11pt]{amsart}

\usepackage{amssymb}
\usepackage{amsmath}
\usepackage{amsfonts}
\usepackage{graphicx}
\usepackage{graphics}
\usepackage{amsbsy}
\usepackage{amscd}
\usepackage{enumerate}

\usepackage{color}

\usepackage[all,knot,poly]{xy}
\usepackage{ifpdf}
\usepackage{ucs}
\usepackage[utf8x]{inputenc}
\usepackage[T1]{fontenc}
\usepackage{latexsym}
\usepackage{array}
\usepackage{enumerate}

\usepackage{pb-diagram}
\usepackage{pb-xy}
\CompileMatrices

\def\TC{\mathrm{TC}}
\def\cat{\mathrm{cat}}
\newcommand{\secat}{\mathrm{secat}}
\def\P{\mathrm{P}}

\def\gd{\mathrm{gd}}
\def\imm{\mathrm{imm}}
\def\N{|\overline{n}|}
\def\immto{\looparrowright}
\def\ppn{\P_{\overline{n}}}
\def\ppm{\P_{\overline{m}}}
\def\xpn{\xi_{\overline{n}}}

\def\spn{S_{\overline{n}}}

\newcommand{\R}{\mathbb{R}}
\newcommand{\Z}{\mathbb{Z}}

\newcommand{\co}{\colon\thinspace}

\newcommand{\p}{\mathrm{P}}

\newcommand{\xx}{\overline{x}}
\newcommand{\yy}{\overline{y}}
\newcommand{\m}{\overline{m}}
\newcommand{\n}{\overline{n}}

\newtheorem{teorema}{Theorem}[section]
\newtheorem{proposicion}[teorema]{Proposition}
\newtheorem{lema}[teorema]{Lemma}
\newtheorem{definicion}[teorema]{Definition}
\newtheorem{corolario}[teorema]{Corollary}
\newtheorem{nota}[teorema]{Remark}

\newtheorem{ejemplo}[teorema]{Example}

\title[Motion planning in projective product spaces]{Topological complexity of motion planning in projective product spaces}

\author[Gonz\'alez, Grant, Torres-Giese, and Xicot\'encatl]{Jes\'us Gonz\'alez, Mark Grant, Enrique Torres-Giese, and Miguel Xicot\'encatl}

\thanks{The first author was supported by Conacyt Grant 102783 during the time this research was conducted.}

\keywords{Topological complexity, product projective spaces, Euclidean immersions of manifolds, generalized axial maps.}

\subjclass[2010]{55M30, 57R42, 68T40.}

\date{\today}

\begin{document}

\begin{abstract}
We study Farber's topological complexity (TC) of Davis' projective product spaces (PPS's). We show that, in many non-trivial instances, the TC of  PPS's coming from at least two sphere factors is (much) lower than the dimension of the manifold. This is in high contrast with the known situation for (usual) real projective spaces for which, in fact, the Euclidean immersion dimension and TC are two facets of the same problem. Low TC-values have been observed for infinite families of non-simply connected spaces only for H-spaces, for finite complexes whose fundamental group has cohomological dimension not exceeding $2$, and now in this work for infinite families of PPS's. We discuss general bounds for the TC (and the Lusternik-Schnirelmann category) of PPS's, and compute these invariants for specific families of such manifolds. Some of our methods involve the use of an equivariant version of TC. We also give a characterization of the Euclidean immersion dimension of PPS's through generalized concepts of axial maps and, alternatively, non-singular maps. This gives an explicit explanation of the known relationship between the generalized vector field problem and the Euclidean immersion problem for PPS's.
\end{abstract}

\maketitle
\tableofcontents

\section{Introduction and notation}\label{intro}

As shown in~\cite{FTY}, the topological complexity (TC) and the Euclidean immersion dimension (Imm) of the $n$-dimensional real projective space $\P^n$  are related by
\begin{equation}\label{classic}
\TC(\P^n)=\mathrm{Imm(\P^n)}-\epsilon(n)=2n-\delta(n)
\end{equation}
where
$$\epsilon(n)=\begin{cases}1,&n=1,3,7;\\0,&\mbox{otherwise},\end{cases}
$$
$\delta(n)=\mathrm{O}(\alpha(n))$, and $\alpha(n)$ denotes the number of ones in the binary expansion of~$n$. It is natural to ask whether the nice phenomenon in the first equality in~(\ref{classic}) is part of a general property of manifolds. Not only does this question have a negative answer, but even close relatives of real projective spaces fail to satisfy the first equality in~(\ref{classic}). For instance, in view of~\cite{AsDaGo} and~\cite{G-robotics}, the failure holds for lens spaces whose fundamental group has torsion of the form $2^e$ for $e>1$. The same answer is observed in a forthcoming paper by two of the authors in which they
study flag manifolds whose fundamental group is an elementary 2-group of rank greater than $1$. This paper now shows that the list of counterexamples extends to Davis' projective product spaces, a family of manifolds giving a rather natural generalization of real projective spaces, and which, in particular, have $\mathbb{Z}_2$ as their fundamental groups (in the `generic' case). Indeed, Theorem~\ref{cota1} in this paper shows that, in contrast to the second equality in~(\ref{classic}), the topological complexity of a projective product space coming from at least two sphere factors can be much lower than the dimension of the manifold. Thus, in those cases, more than half the homotopy obstructions in the motion planning problem for $\P_{\overline{n}}$ are trivial (cf.~Remark~\ref{lasobs} and the considerations after Theorem~\ref{cota1}). Up to the authors knowledge, this gives the first infinite family of non-simply connected closed manifolds which are not H-spaces and whose TC is lower than their dimension (cf.~\cite{CV,LS}; the upper bound in~\cite[Theorem~3]{CF} should be noted, too).

\medskip
In the rest of this introductory section we set up notation and recall needed preliminary results. We use the reduced form of the Schwarz genus (also called sectional category, and denoted by $\secat$) of a fibration, i.e.\ a trivial fibration has zero genus. In particular, we consider the reduced form of the Lusternik-Schnirelmann category (cat) and that of Farber's topological complexity (TC) of a space $X$---the latter being the reduced Schwarz genus of the double evaluation map $X^{[0,1]}\to X\times X$ sending a path $\gamma\colon[0,1]\to X$ to the pair $(\gamma(0),\gamma(1))$. Thus, $\cat(X)=\TC(X)=0$ for a contractible space~$X$. We will also assume the reader is familiar with~\cite{davispps}, and we next briefly recall the required results from that paper.

\medskip
We let $\overline{n}$ stand for an $r$-tuple $(n_1,\ldots,n_r)$ of positive integers with $n_1\leq\cdots\leq n_r$. We consider the diagonal action of $\mathbb{Z}_2$ on $S_{\overline{n}}:=S^{n_1}\times\cdots\times S^{n_r}$, and let $\ppn$ denote the resulting orbit space (so $\P_{(n_1)}$ is the usual real projective space $\P^{n_1}$). We set $\N:=\dim(\ppn)=\dim(\spn)=\sum n_i$ and $\ell(\overline{n})=r$. The real line bundle associated to the obvious covering $\spn\to\ppn$, denoted by $\xpn$ and called the canonical line bundle over $\ppn$, can be used to identify the stable class of the tangent bundle $\tau_{\ppn}$ since
\begin{equation}\label{tangentstableclass}
\tau_{\ppn}\oplus r\varepsilon\approx(\N+r)\xpn.
\end{equation}
Here $\varepsilon$ stands for a trivial line bundle.

\smallskip
The total space of the $k$-fold iterated Whitney sum of $\xpn$ is given by the Borel construction $k\xpn=\spn\times_{\mathbb{Z}_2}\mathbb{R}^k$. In particular, the projectivization of $k\xpn$ is given by
\begin{equation}\label{ksuma}
P(k\xpn)=\ppn\times\P^{k-1}.
\end{equation}

The diagonal inclusion $S^{n_1}\hookrightarrow\spn$ and the projection onto the first factor $\spn\to S^{n_1}$ induce corresponding maps $j\colon\P^{n_1}\hookrightarrow\ppn$ and $p\colon\ppn\to\P^{n_1}$ satisfying
\begin{equation}\label{hopfpreservados} 
j^*(\xpn)\approx\xi_{n_1}\mbox{, \ \ }p^*(\xi_{n_1})\approx\xpn\mbox{, \ \ and \ \ }p\circ j=\mathrm{Id}.
\end{equation}

For $2\leq i\leq r$ there are mod 2 cohomology classes $x_i$ in $\ppn$ with $\dim(x_i)=n_i$ such that the mod 2 cohomology ring of $\ppn$ is given by
\begin{equation}\label{mod2cohring}
H^*(\ppn;\mathbb{Z}_2)=H^*(\P^{n_1};\mathbb{Z}_2)\otimes\Lambda[x_2,\ldots,x_r]
\end{equation}
(where $\Lambda$ denotes an exterior algebra) with the only exception that, if $n_1$ is even, then $x_i^2=x^{n_1}x_i$ whenever $n_i=n_1$. Here $x\in H^1(\P^{n_1};\mathbb{Z}_2)$ satisfies\begin{equation}\label{xsgrandesrestringentrivial}
\mbox{$x=w_1(\xpn),\;$  but all classes $x_i$ restrict trivially under the inclusion $j$.}
\end{equation}

We also need the concept of \hspace{.4mm}``generalized axial map'' as defined in~\cite{AsDaGo}: For a real vector bundle $\alpha$ over a space $X$, we let $S(\alpha)$ and $P(\alpha)$ stand, respectively, for the sphere and projectivized bundles associated to $\alpha$. Let $h_\alpha$ denote the Hopf line bundle over $P(\alpha)$ splitting off $\pi^*(\alpha)$, where $\pi\colon P(\alpha)\rightarrow X$ is the projection. A Hopf-type map\footnote{This is called an `axial map' in~\cite{AsDaGo}, but we have to modify the name in view of Definition~\ref{axialchico} in the next section.} for $\alpha$ is any continuous map $P(\alpha)\rightarrow \P^N$ for which the composite $P(\alpha)\to \P^N\hookrightarrow \P^\infty$ classifies $h_\alpha$. In particular,~(\ref{ksuma}) allows us to talk about Hopf-type maps defined on products of the form $\ppn\times\P^s$.

\section{Immersion dimension}\label{immaxi}
\subsection{Axial maps} Consider a pair of sequences $\overline{n}=(n_1,\ldots,n_r)$ and $\overline{m}=(m_1,\ldots,m_s)$.
\begin{definicion}\label{axialchico}
A continuous map $\alpha\colon\P_{\overline{n}}\times\P_{\overline{m}}\to\P^\infty$ is said to be axial if its restriction to each of the axes classifies the corresponding canonical bundle. By~$(\ref{xsgrandesrestringentrivial})$ this means that $\alpha$ corresponds to the class $x\otimes1+1\otimes x$. A continuous map $\P_{\overline{n}}\times\P_{\overline{m}}\to\P^L$ is called axial if the composite $\P_{\overline{n}}\times\P_{\overline{m}}\to\P^L\hookrightarrow\P^\infty$ is axial.
\end{definicion}

\begin{nota}\label{tcchico} By~$(\ref{hopfpreservados})$, the existence of an axial map $\ppn\times\ppm\to\P^L$ depends only on $n_1$ and $m_1$. In particular, according to~\cite{FTY}, if $n_1=m_1$, $\TC(\P^{n_1})$ is the minimal integer $L$ for which there is an axial map $\ppn\times\ppm\to\P^L$. In any case, an axial map $\ppn\times\ppm\to\P^L$ can exist only if $L\geq\max\{n_1,m_1\}$.
\end{nota}

A slightly weaker concept of axiality arises by requiring that the restriction of $\alpha\colon\P_{\overline{n}}\times\P_{\overline{m}}\to\P^\infty\hspace{.2mm}$ to $\hspace{.5mm}j\times j\colon\P^{n_1}\times\P^{m_1}\hookrightarrow\ppn\times\ppm\hspace{.4mm}$ is axial in the usual sense. Yet, nothing is lost with respect to the more restrictive Definition~$\ref{axialchico}$ if we only care (as we will in this subsection) about the {\it existence} of such maps. Indeed, in view of~$(\ref{mod2cohring})$, the only potential problem arises when $n_2=1$ or $m_2=1$.  To fix ideas, assume $n_2=\cdots=n_\ell=1<n_{\ell+1}$ $(\ell\leq r)$. Then, the restriction of $\alpha$ to its first axis might conceivably correspond to a class of the form $x+\sum_{i=2}^\ell\epsilon_ix_i$. Although such a situation is perfectly attainable, it can be easily fixed. Indeed,~\cite[Theorem~2.20]{davispps} asserts that, under the present conditions, $\ppn$ is homeomorphic to $(S^1)^{\ell-1}\times\P_{\overline{q}}$ where $\overline{q}=(1,n_{\ell+1},\ldots,n_r)$. Thus, 
unless $m_2=1$ (in which case the following adjustment would have to be made on the second axis, too), the required axial map is given by the composite $\P_{\overline{n}}\times\ppm\to\P_{\overline{n}}\times\ppm\overset{\alpha\;}\rightarrow\P^\infty$ where the first map is $\gamma\times1$, and $\gamma$ is the projection $\P_{\overline{n}}\to\P_{\overline{q}}\hspace{.3mm}$ followed by the inclusion $\P_{\overline{q}}\hookrightarrow\ppn$.

\medskip
As a consequence of Remark~\ref{tcchico}, the nice relationship between $\TC$ and the existence of suitable axial maps between (usual) real projective spaces cannot hold for a $\ppn$ with $\ell(\overline{n})>1$. Yet, the axial map approach can be used to characterize the immersion dimension of $\ppn$ in a suitable range of dimensions. Indeed, the following are standard consequences of~\cite{davispps} and \cite{sanderson}:

\smallskip
\begin{enumerate}[(I)]
\item\label{laaxialcritic}The existence of a smooth immersion $\P_{\overline{n}}\immto\mathbb{R}^M$ implies the existence of an axial map $\P^{n_1}\times\P^{\N+r-1}\to\P^{M+r-1}$.

\smallskip
\item\label{ms}The converse of~(\ref{laaxialcritic}) holds provided $\ppn$ is not stably parallelizable and $n_1<2(M-\N)$.
\end{enumerate}

\medskip
We will now elaborate on the previous facts from a purely `projective-product' viewpoint---without relying on the connection through the generalized vector field problem.

\begin{proposicion}\label{immgivesaxi}
The existence of an immersion $\P_{\overline{n}}\immto\mathbb{R}^M$ implies the existence of a Hopf-type map $\P_{\overline{n}}\times\P^{\N+r-1}\to\P^{M+r-1}$.
\end{proposicion}
\begin{proof}
Let $\varepsilon$ be the trivial line bundle over $\ppn$ and $\nu$ the normal bundle of the given immersion. From~(\ref{tangentstableclass}) we have the composite
$$(\N+r)\xi_{\overline{n}}\hookrightarrow(\N+r)\xi_{\overline{n}}\oplus\nu=(\tau_{\P_{\overline{n}}}\oplus r\varepsilon)\oplus\nu=(M+r)\varepsilon.
$$
The required  Hopf-type map is given by the composite $$\P_{\overline{n}}\times\P^{\N+r-1}{\hspace{.3mm}=\hspace{.3mm}}P((\N+r)\xi_{\overline{n}}){\hspace{.3mm}\hookrightarrow\hspace{.3mm}} P\left((M+r)\varepsilon\right){\hspace{.3mm}=\hspace{.3mm}}\ppn\times\P^{M+r-1}\hspace{-.3mm}\stackrel{\mathrm{proj}}\longrightarrow\P^{M+r-1}$$
(cf.~\cite[Section~2]{AsDaGo}).
\end{proof}

\begin{nota}\label{radial}
The converse of Proposition~$\ref{immgivesaxi}$ can be proved (under an additional hypothesis) in terms of a standard application of Haefliger-Hirsch homotopy approximation of monomorphisms by skew-maps in the meta\-stable range~$($\cite{HH},  compare with~\cite{AdGiJa} or \cite[Corollary~2.8]{AsDaGo}$)$. Indeed, the axial map in the conclusion of Proposition~$\ref{immgivesaxi}$ is double covered by a $\mathbb{Z}_2$-equivariant map
$$
S((\N+r)\xi_{\overline{n}})=S_{\overline{n}}\times_{\mathbb{Z}_2}S^{\N+r-1}\to S^{M+r-1}.
$$
This and the projection $(\N+r)\xi_{\overline{n}}\to\ppn$ determine a map $S((\N+r)\xi_{\overline{n}})\to\ppn\times S^{M+r-1}$ which, after radial extension, yields a skew map $(\N+r)\xi_{\overline{n}}\to(M+r)\varepsilon\hspace{.5mm}$ over $\ppn$. Theorem~$1.2$ in~\cite{HH} claims that the latter map can be skew-deformed to a bundle monomorphism $\phi\colon(\N+r)\xi_{\overline{n}}\hookrightarrow(M+r)\varepsilon$ provided
\begin{equation}\label{rmetastable}
3\N<2M.
\end{equation}
Coker$(\phi)$ is then an $(M-\N)$-dimensional bundle which, after taking into account~$(\ref{tangentstableclass})$ and cancelling $r$ trivial sections, yields an isomorphism $\tau_{\ppn}\oplus\mbox{Coker}(\phi)=M\varepsilon$. Thus \cite{hirsch} asserts that Coker$(\phi)$ is the normal bundle of an immersion, as required.
\end{nota}

Of course, the hypothesis~(\ref{rmetastable}) is much stronger than the arithmetical condition in~(\ref{ms}), a hypothesis where $n_1$ plays a more relevant role (and which is in accordance to Remark~\ref{tcchico}).

\begin{proposicion}\label{hopfaxial}
There is a Hopf-type map $\P_{\overline{n}}\times\P^{\N+r-1}\to\P^{M+r-1}$ if and only if there is an axial map $\P^{n_1}\times\P^{\N+r-1}\to\P^{M+r-1}$.
\end{proposicion}
\begin{proof} In view of~(\ref{xsgrandesrestringentrivial}), it suffices to check that the map $\ppn\times\P^{\N+r-1}\to\P^\infty$ that classifies the Hopf line bundle $h_{(\N+r)\xpn}$ corresponds to $x\otimes1+1\otimes x$. For this purpose, we may assume without loss of generality that the given Hopf-type map arises from an immersion as in Proposition~\ref{immgivesaxi} (say for a large enough $M$---this is irrelevant for the intended goal). Then, with the notation of that result, we see from~(\ref{hopfpreservados}) that, by restricting the isomorphism $(\N+r)\xi_{\overline{n}}\oplus\nu=(M+r)\varepsilon$ under the inclusion $j\colon \P^{n_1}\hookrightarrow\ppn$, we get a Hopf-type map
$$
\P^{n_1}\times\P^{\N+r-1}=P((\N+r)\xi_{n_1})\hookrightarrow P((\N+r)\xi_{\overline{n}})=\P_{\overline{n}}\times\P^{\N+r-1}\to\P^{M+r-1}
$$
which, as proved in~\cite{AdGiJa}, must also be an axial map. Thus,~(\ref{xsgrandesrestringentrivial}) implies that $h_{(\N+r)\xpn}$ corresponds, under the identification $P((\N+r)\xpn)=\ppn\times\P^{\N+r-1}$ in~(\ref{ksuma}), to a class of the form $$1\otimes x+(x+\sum\mu_ix_i)\otimes1$$ where the summation runs over indexes $i$ with $n_i=1$, and each $\mu_i$ is either 0 or 1. But the first isomorphism in~(\ref{hopfpreservados}) and the naturality of the construction of Hopf line bundles imply $\mu_i=0$ for all relevant $i$.
\end{proof}

\begin{ejemplo}\label{kee}
The arithmetical hypothesis in~\emph{(\ref{ms})} is superfluous when $\ell(\overline{n})=1$, but it is needed if $\ell(\overline{n})>1$. From our perspective, such a phenomenon is due to the fact that, although the immersion dimension of any standard real projective space holds within Haefliger's metastable range~$($\cite{AdGiJa}$)$, as noted in~\cite{davispps}, a projective product space $\ppn$ with $\ell(\overline{n})>1$ usually admits (very) low-codimension Euclidean immersions---compare to Remark~$\ref{aclara}$ below. For instance\footnote{We thank Kee Lam for kindly pointing out this example.}, the non-parallelizable $\P_{(12,14)}$ does not immerse in $\mathbb{R}^{30}$ in view of~\cite[Theorem~3.4]{davispps}, \cite[Lemma~2.2]{sanderson}, and~\cite{KD} (in that order), but the existence of the corresponding axial map in~\emph{(\ref{laaxialcritic})} is obtained in~\cite{KD} through a Postnikov tower argument.
\end{ejemplo}

Despite Example~\ref{kee}, the method of proof of the main result in~\cite{AdGiJa} yields:
\begin{proposicion}\label{agj}
If $\gd\left(-(\N+r)\xi_{n_1}\right)>\lceil{(n_1+1)/2}\rceil$, then the arithmetical hypothesis in~\emph{(\ref{ms})} is superfluous.
\end{proposicion}
\begin{proof}
Assume for a contradiction that, for some $M$, there is an axial map $\P^{n_1}\times\P^{\N+r-1}\to\P^{M+r-1}$ but that the non-stably parallelizable $\ppn$ does not immerse in $\mathbb{R}^M$. Without loss of generality we can assume $M=\imm(\ppn)-1>\N$. Then,~\cite[Theorem~3.4]{davispps} gives
$$
M-\N=\imm(\ppn)-\N-1=\gd(-(\N+r)\xi_{n_1})-1\geq\left\lceil{\frac{n_1+1}2}\right\rceil,
$$
which amounts to having the arithmetical hypothesis in~(\ref{ms}).\end{proof}

\begin{nota}\label{aclara}
In the same line of reasoning as in Example~$\ref{kee}$, it follows from~\cite{adams} that, for any large $n_1$, there are instances of spaces $\P_{(n_1,\ldots,n_r)}$ for which the hypothesis in Proposition~$\ref{agj}$ fails.
\end{nota}

It is worth mentioning that, for $n_1\leq 9$, the arithmetical hypothesis in~(\ref{ms}) above is superfluous\footnote{Kee Lam has brought the author's attention that the smallest case where the arithmetical hypothesis in~(\ref{ms}) is actually needed takes place when $n_1=10$. 
}.
As in the proof of Proposition~\ref{agj}, such an assertion can be verified by checking that, in the indicated range, there is no axial map $\P^{n_1}\times\P^{\N+r-1}\to\P^{M+r-1}$ with $M=\imm(\ppn)-1>\N$. Indeed, under the current hypothesis, such an axial map is prevented by the relation
\begin{equation}\label{hopf}
(x+y)^{M+r}\neq0
\end{equation}
where $x$ and $y$ denote respectively the generators of the mod~$2$ cohomology groups $H^1(\P^{n_1};\mathbb{Z}_2)$ and $H^1(\P^{\N+r-1};\mathbb{Z}_2)$. Explicitly, the basis element $x^gy^{\N+r-1}\in H^*(\P^{n_1}\times\P^{\N+r-1};\mathbb{Z}_2)$ appears in the expansion of~$(\ref{hopf})$ with coefficient
\begin{equation}\label{bino}
\binom{\N+r+g-1}{g}
\end{equation}
where $g=\gd(-(\N+r)\xi_{n_1})$. But, under the current hypothesis, a direct verification using~\cite{keesectioning} (or, alternatively, Table~$4.4$ and Proposition~$4.5$ in~\cite{davispps}) shows that~$(\ref{bino})$ is odd. For instance, consider the case $n_1=6$, where the assumption that $\ppn$ is not stably parallelizable means $\N+r\not\equiv0\bmod8$. Then,~\cite[Table~4.4]{davispps} gives $g=(6,6,5,4,3,2,1)$ for $\N+r\equiv(1,2,3,4,5,6,7)\bmod8$. So
$$
\binom{\N+r+g-1}{g}\equiv\left(\binom{6}{6},\binom{7}{6},\binom{7}{5},\binom{7}{4},\binom{7}{3},\binom{7}{2},\binom{7}{1}\right)\equiv1\;\bmod2.
$$


\medskip
We close this subsection by remarking that, just as the situation in Example~\ref{kee} for the condition $n_1<2(M-\N)$, the hypothesis that $\ppn$ is not stably parallelizable is also needed in~(\ref{ms}). Yet, the full TC-axial picture is well understood in the stably parallelizable case. In fact, the situation is entirely similar to that in the classical case with $\ell(\overline{n})=1$, where there are well-known axial maps $\P^n\times\P^n\to\P^n$
for $n=1,3,7$, but of course no immersion $\P^n\immto\mathbb{R}^n$. Namely, since the immersion dimension of a stably parallelizable $\ppn$ is $\N+1$, there is an axial map $\P_{\overline{n}}\times\P^{\N+r-1}\to\P^{\N+r}$. But there is a finer (and optimal) axial map
\begin{equation}\label{finer}
\P_{\overline{n}}\times\P^{\N+r-1}\to\P^{\N+r-1}
\end{equation}
(which cannot come from an immersion). Indeed, as shown in~\cite{davispps}, the stable parallelizability of $\ppn$ means that the exponent in the highest $2$-power dividing $\N+r$ is no less than $\phi(n_1)$---the number of positive integers less than or equal to $n_1$ and which are congruent to $0$, $1$, $2$, or $4$ mod $8$. Therefore, classical work of Hurwitz, Radon, and Eckmann on the so-called Hurwitz-Radon matrix equations gives in fact a non-singular bilinear map $\mathbb{R}^{n_1+1}\times\mathbb{R}^{\N+r}\to\mathbb{R}^{\N+r}$ and, in view of Remark~\ref{tcchico}, an axial map of the form~(\ref{finer}). An intriguing possibility is that explicit `linear' formul\ae\ leading to an axial map~(\ref{finer}) could be deduced from a refinement of the Clifford-algebra input in the Hurwitz-Radon number---without relying on Remark~\ref{tcchico}.

\subsection{Non-singular maps}
The existence of axial maps can be translated into the existence of certain non-singular maps. Not only is such a fact a straightforward generalization of the corresponding well-known property for usual projective spaces, but the language of non-singular maps turns out to be irrelevant for the purposes of the paper, since they fail to provide local motion planners as in the classical case. Consequently, these ideas are loosely treated in this subsection, mentioned only for completeness purposes.

\smallskip
There are two closely related notions of non-singular maps associated to an axial map between projective product spaces. In the first one, for an $\ell$-tuple $\overline{q}=(q_1,\ldots,q_\ell)$, we consider the cone $Q_{\overline{q}}$ in $\R^{q_1+1} \times \dots \times \R^{q_\ell+1} $  consisting of tuples $\xx=(x_1, \ldots, x_\ell)$ with $|x_1| = \cdots = |x_\ell|$. Thus, $\p_{\overline{q}}$ is the projectivization of $Q_{\overline{q}}$, i.e.~$\p_{\overline{q}}$ is the subspace of $\P^{|\overline{q}|+\ell-1}$ consisting of the lines contained in $Q_{\overline{q}}$. Then, a continuous map $f:Q_{\overline{n}}\times Q_{\overline{m}} \to \mathbb{R}^{k+1}$ is said to be non-singular if $f(\lambda \overline{x}, \mu \overline{y})= \lambda\mu f(\overline{x},\overline{y})$ for $\lambda,\mu\in\mathbb{R}$, and if the equality $f(\overline{x},\overline{y})=0$ holds only with $\overline{x}=0$ or $\,\overline{y}=0$. With this definition, {\it there is a one-to-one correspondence between the set of non-singular maps $f:Q_{\overline{n}}\times Q_{\overline{m}}\to\mathbb{R}^{k+1}$ (taken up to multiplication by a non-zero scalar) and the set of axial maps $g:\ppn\times\ppm\to\P^k$.} Such a corresponding pair $(f,g)$ fits in a commutative diagram
$$\xymatrix{
Q_{\n} \times Q_{\m} \ar[d]_f &
S_{\n}  \times   S_{\m} \ar[d]_{f'} \ar@{_(->}[l] \ar[r] &
S_{\n}  \times_{\Z_2}   S_{\m} \ar[d]_h \ar[r] &
\p_{\n}  \times   \p_{\m} \ar[d]_g \\
\R^{k+1} &
\R^{k+1} - \{0\} \ar@{_(->}[l] \ar[r]^{\ \ \ \ \ \rho}&
S^k \ar[r]&
\P^k.
}$$
Here the unlabelled horizontal maps facing east are the obvious two fold coverings, $\rho$ is the normalization map $\rho(u)=u/|u|$, $f'$ is the restriction $f|_{S_{\overline{n}}\times S_{\overline{m}}}$, and the right hand square is a pullback (hence $h$ is $\mathbb{Z}_2$-equivariant). Explicitly, given $f$, $g([\overline{x}],[\overline{y}])$ is the line in $\mathbb{R}^{k+1}$ that goes through the origin and $f(\overline{x},\overline{y})$. Conversely, given $g$, pick $h$ as in the diagram above and precompose it with the double covering $S_{\overline{n}}\times S_{\overline{m}}\to S_{\overline{n}}\times_{\mathbb{Z}_2} S_{\overline{m}}$ to get a $\mathbb{Z}_2$-biequivariant map $\tilde{g}:S_{\overline{n}}\times S_{\overline{m}} \to S^{k}$. Then $f$ is the ``bi-radial'' extension of $\tilde{g}$ given by
\begin{equation}\label{alaXico}
f(\xx, \yy) = 
\begin{cases}
\frac{|\xx|}{\sqrt{r}} \frac{|\yy|}{\sqrt{s}} \;  \tilde g\left(\frac{\sqrt{r}}{|\xx|} \,  \xx, \frac{\sqrt{s}}{|\yy|}  \, \yy \right), & 
\text{if $\xx \neq 0$  and $\yy \neq 0;$} \\
0,&  \text{if  $\xx =0$    or $\yy=0$}.
\end{cases}
\end{equation}

Note that if $f:\mathbb{R}^{n_1+1} \times \mathbb{R}^{m_1+1} \to \mathbb{R}^{k+1}$ is a non-singular map (in the usual sense), then for any
$\overline{n}=(n_1,n_2,\ldots,n_r)$ and $\overline{m}=(m_1,m_2,\ldots,m_s)$ a non-singular map $Q_{\overline{n}}\times Q_{\overline{m}}\to\R^{k+1}$ can be defined by
$(\overline{x},\overline{y})\mapsto f(x_1,y_1)$. Of course, this fact is compatible with Remark~$\ref{tcchico}$.

\smallskip
A slight variation of the notion of non-singular maps goes as follows: Set $V_{\overline{t}}=\mathbb{R}^{t_1 +1}\times \cdots \times \mathbb{R}^{t_\ell+1}$. A map $f:V_{\overline{n}}\times V_{\overline{m}} \to \mathbb{R}^{k+1}$ is said to be non-singular if $f(\lambda \overline{x}, \mu \overline{y})= \lambda\mu f(\overline{x},\overline{y})$ for $\lambda,\mu\in\mathbb{R}$, and if the equality $f(\overline{x},\overline{y})=0$ holds only when a coordinate $x_i$ of $\overline{x}$  or a coordinate $y_j$ of $\overline{y}$ vanishes. Then the above considerations apply basically without change, except that~(\ref{alaXico}) takes the slightly more elaborated form
$$
f(\overline{x}, \overline{y}) =
  \begin{cases}
N(\overline{x},\overline{y})\;\tilde{g}\left( \frac{x_1}{|x_1|},\ldots,\frac{x_r}{|x_r|},\frac{y_1}{|y_1|},\ldots,\frac{y_r}{|y_s|}  \right), 
& \text{if no } x_i \text{ nor } y_j\text{ is zero}; \\
   0,      & \text{otherwise},
  \end{cases} 
$$
where $N(\overline{x},\overline{y})=\left(|x_1|\cdots|x_r|\right)^{\frac{1}{r}}\left(|y_1|\cdots|y_s|\right)^{\frac{1}{s}}$.

\section{Topological complexity}\label{SecTC}

In this section we give several general estimates for $\TC(\ppn)$. We find that $\TC(\ppn) < \dim(\ppn)$ in certain cases, indicating that a simple relation to immersion dimension such as~(\ref{classic}) does not hold for these manifolds. We also compute the exact value of $\TC(\ppn)$ in many cases (Proposition~\ref{analogo}), and give evidence toward the appealing possibility that $\TC(\ppn)$ would depend mostly on $\TC(\P^{n_1})$ and $\ell(\overline{n})$.

\smallskip
Let $\overline{\infty}$ stand for the $r$-tuple $(\infty,\ldots,\infty)$, and let $\P_{\overline{\infty}}$ denote the quotient of $\prod_{r}S^\infty$ by the diagonal action of $\mathbb{Z}_2$ (with the antipodal action on each factor). Note that $\P_{\overline{\infty}}$ is an Eilenberg-MacLane space $K(\mathbb{Z}_2,1)$ containing $\ppn$.

\begin{lema}\label{cells}
There is a CW decomposition for $\P_{\overline{\infty}}$ whose $n_1$-skeleton is contained in $\ppn$.
\end{lema}
\begin{proof}
Let $e^0_+\cup e^0_-\cup\cdots\cup e^m_+\cup e^m_-$ be the usual $\mathbb{Z}_2$-equivariant cell structure on a sphere $S^m$,  and consider the resulting product structure
\begin{equation}\label{celstructure}
\spn=\bigcup e^{i_1}_\pm\times\cdots\times e^{i_r}_\pm.
\end{equation}
If $\tau$ stands for the generator of $\mathbb{Z}_2$, then
a cell structure on $\ppn$ can be formed by identifying a cell $e^{i_1}_\pm\times\cdots\times e^{i_r}_\pm$ in~(\ref{celstructure}) with the corresponding cell $\tau\cdot(e^{i_1}_\pm\times\cdots\times e^{i_r}_\pm)$. If $\ell(\overline{m})=\ell(\overline{n})$ and $n_i\le m_i$, the inclusion $\ppn\hookrightarrow\ppm$ contains the $n_1$-skeleton of $\ppm$. Thus the required cell structure in $\P_{\overline{\infty}}$ is the inductive one under the above inclusions.\end{proof}

We are indebted to Sergey Melikhov for pointing out (in~\cite{Mel}) the  proof of the following fact:

\begin{proposicion}\label{polyhedral_fibers} Let $M^m$ and $N^n$ be closed smooth manifolds, and let $C^\infty(M,N)$ denote the space of smooth maps in the Whitney $C^\infty$-topology. Then for $f\co M\to N$ in a dense subset of $C^\infty(M,N)$, the fibers $f^{-1}(y)$ with $y\in N$ are all polyhedra of dimension $\le \min(m-n,0)$.
\end{proposicion}
\begin{proof}
First we note that the set of triangulable maps is dense in $C^\infty(M,N)$. Recall that a smooth map $f\co M\to N$ is {\em triangulable} if there exists a PL map $g\co K\to L$ between PL manifolds, and homeomorphisms $h\co M\to K$ and $h'\co N\to L$ such that $g\circ h = h'\circ f$. By Verona's proof of Thom's triangulation conjecture \cite{Ver}, we know that all proper, topologically stable maps $f\co M\to N$ are triangulable. By the Thom-Mather theorem (a full proof of which appears in \cite{GWPL}), such maps form an open dense subset of $C^\infty(M,N)$.

Next, we note that the fibers $f^{-1}(y)$ of a triangulable map $f\co M\to N$ are all polyhedra (they are homeomorphic to simplicial complexes). For given $y\in N$, we may choose a triangulation $h'\co N\to L$ as above with $h'(y)$ a vertex of $L$. Then $f^{-1}(y)$ is homeomorphic with $g^{-1}(h'(y))$, a subcomplex of $K$. 

Finally, we claim that for $f\co M\to N$ in an open dense subset of the space $C^\infty(M,N)$, the fibers $f^{-1}(y)$ all have covering dimension $\le \min(m-n,0)$. Intersecting this set with the set of proper, topologically stable maps, we find an open dense set of maps whose fibres are all polyhedra of covering dimension $\le \min(m-n,0)$. Since covering dimension is a topological property, this proves the Proposition.

The proof of the final claim follows from the multi-jet transversality theorem \cite{GG}, which implies that for an open dense set of mappings $f\co M\to N$, the fibers $f^{-1}(y)$ all have the structure of a smooth submanifold of $M$ of dimension $\min(m-n,0)$ away from at most finitely many isolated singular points.
\end{proof}

\begin{teorema}\label{main2}
$\TC(\ppn)\leq2\N-n_1+1$ for $\hspace{.3mm}\ell(\overline{n})>1$. On the other hand, the following numbers are equal, giving a lower bound for $\TC(\ppn)\colon$
\begin{itemize}
\item The Schwartz genus of the obvious double cover $S_{\overline{n}}\times_{\mathbb{Z}_2}S_{\overline{n}}\to\ppn\times\ppn$.
\item The smallest integer $L$ for which $(L+1)\xi_{\overline{n}}\otimes\xi_{\overline{n}}$ admits a nowhere zero section.
\item The smallest integer $L$ for which there is an axial map $\ppn\times\ppn\to\P^{L}$.
\item $\TC(\P^{n_1})$
\end{itemize}
\end{teorema}
\begin{proof}
It follows from Remark~\ref{tcchico} and the first two conditions in~(\ref{hopfpreservados}) that the number described in each of the first three items does not change if $\overline{n}$ is replaced by $n_1$ (for the first item we use the fact that the indicated double cover is the sphere bundle associated to $\xpn\otimes\xpn$). Therefore, the equality of the four listed numbers follows from~\cite[Theorem~6.1]{FTY}. The fact that they give a lower bound for $\TC(\ppn)$ follows from the third condition in~(\ref{hopfpreservados}) and the behavior of $\TC$ under retracts.

\smallskip
We use the argument in~\cite[Corollary~4.5]{grant} (which is inspired in turn by~\cite{ow}) to prove the upper bound in this theorem. Set $L=2\N-n_1+1$. By~(\ref{classic}) and Remark~\ref{tcchico}, we can chose an axial map $q\colon\ppn\times\ppn\to\P^L$. Since the axial condition is homotopical, we can assume first that $q$ is smooth and then, by Proposition~\ref{polyhedral_fibers}, that for each $z\in\P^L$ the inverse image $q^{-1}(z)$ is homeomorphic to a CW complex of dimension at most $n_1-1$. Then, the axiality of $q$ implies that the image of the class $x$ in~(\ref{xsgrandesrestringentrivial}) under the composite
\begin{equation}\label{hoty}
q^{-1}(z)\hookrightarrow \P_{\overline{n}} \times \P_{\overline{n}} \stackrel{\pi_i\;\,}\to \P_{\overline{n}}
\end{equation}
is independent of the projection $\pi_i\colon\P_{\overline{n}} \times \P_{\overline{n}} \to \P_{\overline{n}}$ ($i=1,2$) used. In fact, Lemma~\ref{cells} and the dimensionality assumption on $q^{-1}(z)$ imply that the actual homotopy type of~(\ref{hoty}) is independent of $i$. The result then follows from Lemma~2.5 and Theorem~4.3 in~\cite{grant}.
\end{proof}

Of course, part of the argument for the lower bound in Theorem~\ref{main2} actually yields $\TC(\P^{n_1})\leq\TC(\P_{(n_1,n_2)})\leq\cdots\leq\TC(\P_{(n_1,\ldots,n_{r-1})})\leq\TC(\ppn)$.
On the other hand, the argument proving the upper bound uses and corrects the proof of~\cite[Corollary~4.5]{grant} which, instead of using Proposition~\ref{polyhedral_fibers}, is based on an assertion about approximating axial maps by submersions. But such a claim is false in general, as illustrated next.

\begin{ejemplo}\label{mark_example}
Since $\P^2\looparrowright \mathbb{R}^3$, there exists an axial map $q\co\P^2\times\P^2\to\P^3$. Note that $2<3<2\cdot 2 =4$. However, $q$ is not homotopic to a submersion. In fact, there does not exist {\em any} submersion $\P^2\times\P^2\to\P^3$, by the following easy argument involving Stiefel-Whitney classes: Suppose $g\co\P^2\times\P^2\to\P^3$ is a submersion. Then we obtain the short exact sequence of vector bundles over $\P^2\times \P^2$
$$
0 \to E \to T(\P^2\times\P^2) \stackrel{dg\,}{\longrightarrow} g^*T(\P^3) \to 0
$$
where the kernel $E$ is a real line bundle. It then follows that $$w(\P^2\times\P^2) = w(E)\hspace{.6mm}g^*w(\P^3) = w(E)$$ $($the latter equality since $\P^3$ is parallelizable$)$. But this is impossible since, for example, $w_2(\P^2\times\P^2)\neq 0$.
\end{ejemplo}

\begin{nota}\label{CV}
It is possible to prove the upper bound in Theorem~$\ref{main2}$ by applying~\cite[Theorem~3]{CV} to an axial map $\ppn\times\ppn\to\P^{2\N-n_1+1}$, and noticing that the canonical inclusion $\ppn\hookrightarrow\P^{2\N-n_1+1}$ is an $n_1$-equivalence. We have chosen the approach in~\cite{grant} due to the intrinsic interest of Proposition~$\ref{polyhedral_fibers}$.
\end{nota}

\begin{nota}\label{lasobs}
The standard upper bound $\TC\leq2\dim$ means that, 
in general, there are up to twice $\dim(X)$ classical homotopy obstructions to consider when bounding $\TC(X)$ from above. For instance, the first top two are central in~\cite{CF}, with the very top one being critical for Costa-Farber's applications---the next-to-the-top one comes for free from~\cite{berstein}. Thus, the upper bound in Theorem~$\ref{main2}$ is already taking care of the first $n_1-1$ of these obstructions for $X=\ppn$.
\end{nota}

The lower bound in Theorem~\ref{main2} is rather crude, as it ignores information coming from $S^{n_2}\times\cdots\times S^{n_r}$. For
instance,~\cite[Theorem~4.5]{FTY},~(\ref{mod2cohring}), and `zero-divisors' cup-length ($\mathrm{zcl}$) considerations (as defined in~\cite{farberTCplanning}) easily yield
\begin{equation}\label{enriquesbound}
\TC(\ppn)\geq2^{e+1}+\ell({\overline{n}})-2 \mbox{ provided }n_1\geq 2^e,
\end{equation}
which improves by arbitrarily large amount the lower bound in Theorem~\ref{main2} when $\ell(\overline{n})\gg0$. 
%
%
On the other hand, the general philosophy behind~(\ref{classic}) implies that the lower bound in Theorem~\ref{main2} can be much stronger than that in~(\ref{enriquesbound}) if $\ell(\overline{n})=2$. For instance~\cite{james} gives
\begin{equation}\label{jassint}
\TC(\P_{(2^e-1,2^e-1)})\geq\TC(\P^{2^e-1})\geq2^{e+1}-2e-(2,1,1,3)
\end{equation}
provided $e\equiv(0,1,2,3)\bmod4$, a bound which is almost twice that in~(\ref{enriquesbound}). Of course, further results of this sort can be deduced from our current knowledge of the immersion dimension of (usual) real projective spaces. In view of~\cite[Theorems~2.1 and 2.4]{GTV}, it should be possible to use zcl-considerations based on {\it generalized} cohomology theories in order to insert the nice $\ell(\overline{n})$-feature of~(\ref{enriquesbound}) into the lower bound in Theorem~\ref{main2}, thus merging the corresponding strengths of~(\ref{enriquesbound}) and~(\ref{jassint}) into a single lower bound (we hope to explore such a possibility elsewhere). 

\medskip
More interesting is the fact that $\TC(\ppn)$ can be arbitrarily smaller than the dimension of $\ppn$. The simplest of such situations originates from the subadditivity of TC~(\cite{farberTCplanning}), as $\TC(\ppn)-\TC(\ppm)\leq2$ whenever $\TC(\ppn)\approx\ppm\times S^{n_i}$ (the latter decomposition is characterized arithmetically in~\cite[Theorem~2.20]{davispps}). As an extreme situation consider the following partial analogue of~\cite[(2.21)]{davispps}:

\begin{proposicion}\label{analogo}
Let $\phi(n_1)$ be the number of positive integers equal to or less than $n$ which are congruent to $0$, $1$, $2$, or $4$ (mod $8$). If $\nu(n_i + 1)\geq\phi(n_1)$ for all $i>1$, then $$\mathrm{zcl}_{\mathbb{Z}_2}(\P^{n_1})+\ell(\overline{n})-1\leq\TC(\ppn)\leq\TC(\P^{n_1})+\ell(\overline{n})-1.$$ Further, both inequalities above become equalities precisely for $n_1$ a $2$-power. 
\end{proposicion}
\begin{proof}
The first inequality is~(\ref{enriquesbound}); the second inequality follows from~\cite[Theorem~2.20]{davispps}. The final assertion follows from the standard fact that $\TC(\P^{n_1})=\mathrm{zcl}_{\mathbb{Z}_2}(\P^{n_1})$ precisely for $n_1$ a 2-power.
\end{proof}

Proposition~\ref{analogo} suggests the possibility that $\TC(\ppn)$ can be estimated for any $\overline{n}$ in terms of $\TC(\P^{n_1})$ and $\ell(\overline{n})$ alone. Theorem~\ref{cota1} below (whose proof is postponed to the next section) fits into such a general philosophy, and shows that  the low TC-phenomenon in Proposition~\ref{analogo} holds even if there are no spheres factoring out $\ppn$.

\begin{teorema}\label{cota1}
If $k$ denotes the number of spheres $S^{n_i}$ with $n_i$ even and $i>1$, then $\TC(\ppn)<(\TC(\P^{n_1})+1)(\ell(\overline{n})+k)$.
\end{teorema}

The upper bound in Theorem~$\ref{cota1}$ will be much lower than the dimension of $\ppn$ provided the sum $n_2+\cdots+n_r$ is large enough---which can hold even if there are no spheres $S^{n_i}$ factoring out $\ppn$. Thus, in such cases, most of the homotopy obstructions in the motion planning problem for $\ppn$ already vanish. It is worth noticing that $\TC(\ppn)$ is not always less than $\dim(\ppn){:}$ if $1^r$ stands for the $r$-tuple $(1,\ldots,1)$, then $\TC(\P_{1^r})=\dim(\P_{1^r})$, in view of Proposition~$\ref{analogo}$. On the other hand, the upper bound in Theorem~\ref{cota1} not always improves that in Theorem~\ref{main2}. For instance, in the case of $\P_{(2^e,2^e)}$, the former bound is $6\cdot2^e$ while the latter one is only $3\cdot2^e+1$. 

%

\section{Equivariant topological complexity}
In a recent paper \cite{CG} Hellen Colman and the second author explore an equivariant generalization of topological complexity, in the setting of compact group actions. Here we give additional examples and results which will be useful in applying their results to the estimation of topological complexity of projective product spaces.

\smallskip
Let $G$ be a compact Hausdorff topological group (in our present applications $G$ will be the cyclic group $\Z_2$). If $p\co E\to B$ is a $G$-map, the {\em equivariant sectional category} of $p$, denoted $\secat_G(p)$, is defined in \cite[Section 5]{CG} to be the least integer $k$ such that $B$ may be covered by $k$ invariant open sets $U_1,\ldots , U_k$ on each of which there exists a $G$-homotopy section, that is a $G$-map $s\co U_i \to E$ such that $p\circ s$ is $G$-homotopic to the inclusion $i_{U_i}\co U_i\hookrightarrow B$. If $p$ is a $G$-fibration, then this is equivalent to requiring the existence of a $G$-section $s\co U_i\to E$ such that $p\circ s = i_{U_i}$.

\smallskip
In particular, for any $G$-space $X$ the {\em equivariant topological complexity} of $X$ is defined in \cite[Section 6]{CG} to be the equivariant sectional category of the double evaluation map $X^{[0,1]}\to X\times X$. Here $G$ acts diagonally on the product and by composition on the path space of $X$.

\smallskip
In keeping with the conventions in place in this paper, we will define the equivariant topological complexity to be one less than the number of sets in the open cover; thus
\[
\TC_G(X)=\secat_G(X^{[0,1]}\to X\times X)-1.
\]
\begin{lema}\label{tcequivariante}
Let $G=\Z_2$ act antipodally on the sphere $S^n$, where $n\ge 1$. Then
\[
\TC_G(S^n) = \left\{\begin{array}{ll} 1 & \mbox{if $n$ is odd,} \\
                                      2 & \mbox{if $n$ is even.} \end{array}\right.
\]
\end{lema}
\begin{proof}
We argue that the usual motion planning rules on the spheres can be made equivariant with respect to the antipodal action, by choosing vector fields which are equivariant.

Suppose $n$ is odd. Then the projective space $\P^n$ has zero Euler characteristic and so admits a nowhere-vanishing vector field. Using the double cover immersion $S^n\to P^n$, this pulls back to a nowhere-vanishing vector field $v$ on $S^n$ which is equivariant in the sense that $dg(v(A)) = v(gA)$ for $g\in G$ and $A\in S^n$. We consider the open sets $U_0=\{(A,B)\in S^n\times S^n\mid A\neq -B\}$ and $U_1 = \{(A,B)\in S^n\times S^n\mid A\neq B\}$. We define $s_0$ on $U_0$ by choosing the shortest geodesic path from $A$ to $B$ (traveled at constant velocity). We define $s_1$ on $U_1$ in two stages: first travel from $A$ to $-A$ along the great circle in the direction determined by $v(A)$; second travel from $-A$ to $B$ along the shortest geodesic path. It is easy to check that these sets and motion planning rules are $G$-invariant.

When $n$ is even, removing a point $[C]$ from $P^n$ gives an open manifold homotopy equivalent to $P^{n-1}$, which therefore admits a nowhere-vanishing vector field. Again we pull this back to obtain a nowhere-vanishing equivariant vector field $v'$ on $S^n-\{-C,C\}$. We let $U_0$ and $s_0$ be as before. We let $U_1'=\{(A,B)\in S^n\times S^n\mid A\neq B, C,-C\}$ and define $s_1'$ using $v'$ similarly to $s_1$. Finally we let $U_2'=\{(A,-A)\mid A\in W\cup -W\}$, where $W$ is a small open disk neighbourhood centred on $C$. The path $s_2'(A,-A)$ for $A\in W$ travels first along the geodesic segment to the centre $C$ of $W$; then along some fixed path $\gamma$ from $C$ to $-C$; then along the geodesic segment in $-W$ to $-A$. For $A\in -W$ the path $s_2'(A,-A)$ travels first along the geodesic segment in $-W$ to $-C$; then along $-\gamma$ from $-C$ to $C$; then along the geodesic segment in $W$ to $-A$.

The lower bounds are given by the obvious inequality $\TC(X)\le\TC_G(X)$, which holds for any $G$-space $X$.
\end{proof}

\begin{teorema}\label{product} Let $G$ be a compact Lie group, and let $X$ and $Y$ be smooth $G$-manifolds. Then
\[
\TC_G(X\times Y)\le \TC_G(X) + \TC_G(Y)
\]
where $X\times Y$ is given the diagonal $G$-action.
\end{teorema}
\begin{proof}
Let $\TC_G(X)=n$ and $\TC_G(Y)=m$. Suppose that $X\times X = U_0\cup  \cdots \cup U_n$ where the $U_i$ are open invariant sets with $G$-sections $s_i\co U_i\to X^{[0,1]}$. Suppose further that $Y\times Y = V_0\cup  \cdots \cup V_m$ where the $V_j$ are open invariant sets with $G$-sections $\sigma_j\co V_j\to Y^{[0,1]}$. We can find a $G$-invariant partition of unity $\{ f_i\}$ on $X\times X$ subordinate to $\{U_i\}$ (see \cite[Corollary B.33]{GGK}). Likewise let $\{g_j\}$ be a $G$-invariant partition of unity on $Y\times Y$ subordinate to $\{ V_j\}$.

The rest of the proof proceeds by direct analogy with the proof in the non-equivariant case given in \cite[Theorem 11]{farberTCplanning}, hence is omitted.
\end{proof}

\begin{nota} Theorem~$\ref{product}$ is certainly not the most general setting in which the product inequality holds. For instance, we believe it holds whenever $X$ and $Y$ are $G$-ENRs.
\end{nota}

\begin{corolario} Consider the diagonal action of $\hspace{.5mm}\Z_2\hspace{-.3mm}$ on $S_{\overline{n}}=S^{n_1}\times \cdots \times S^{n_r}$. If $k$ denotes the number of spheres with $n_i$ {\em even}, then
\[
\TC_G(S_{\overline{n}}) = \ell(\overline{n}) + k.
\]
\end{corolario}

The main result we will apply from \cite{CG} gives an upper bound for the (non-equivariant) topological complexity of a Borel fibration in terms of the topological complexity of the base and the equivariant topological complexity of the fibre.

\begin{teorema}[{\cite[Theorem 6.21]{CG}}]\label{CGTCG}
Let $X$ be a $G$-space, and let $E\to B=E/G$ be a numerable principal $G$-bundle. Then
\[
\TC(X_G) < (\TC_G(X)+1)(\TC(B)+1),
\]
where $X_G=E\times_G X$ is the corresponding Borel space of $X$.
\end{teorema}

\begin{proof}[Proof of Theorem~$\ref{cota1}$]
Let $\overline{m}=(n_2,\ldots,n_r)$. Note that $\ppn$ can be thought of as the Borel space $S^{n_1}\times_{\mathbb{Z}_2}S_{\overline{m}}$. The result then follows from Theorem~\ref{CGTCG}.
\end{proof}

The argument in the proof of Theorem~\ref{cota1} can be used to give low upper bounds for the LS-category of projective product spaces (extending the phenomenon noted in~\cite[(2.21)]{davispps} when $\ppn$ has a full set of factoring spheres). Namely
\begin{equation}\label{cates}
\cat(\ppn)<(n_1+1)\ell(\overline{n}).
\end{equation}
Since $\TC\leq2\hspace{.3mm}\cat$, we get in particular
\begin{equation}\label{teces}
\TC(\ppn)<2(n_1+1)\ell(\overline{n})-1
\end{equation}
which improves on Theorem~\ref{cota1} only when $\ppn$ comes from a product having `enough' even dimensional spheres, i.e.~$k\geq C_{n_1}\ell(\overline{n})$ where
$$
C_{n_1}=\frac{2n_1+1-\TC(\P^{n_1})}{\TC(\P^{n_1})+1}.
$$
Note that, although $k\leq\ell(\overline{n})$, $C_{n_1}\!\ll1$ for any `generic' $n_1$.

\medskip
Just as~$(\ref{teces})$ and Theorem~$\ref{cota1}$ may fail to improve the upper bound in Theorem~$\ref{main2}$, the bound in~$(\ref{cates})$ is not always useful (for instance if all the $n_i$ are equal). For such cases it is worth keeping in mind that, in view of~\cite[Theorem~3.5]{berstein}, the standard estimate $\cat\leq\dim$ is improved by the inequality $\cat(\ppn)<\dim(\ppn)$ provided $n_1>1$ and $\ell(\overline{n})>1$.

\vspace{1mm}
{\sc Departamento de Matem\'aticas

Centro de Investigaci\'on y de Estudios Avanzados del IPN

M\'exico City 07000, M\'exico}

{\it E-mail address:} {\bf jesus@math.cinvestav.mx}
 
\medskip

 {\sc School of Mathematical Sciences
 
The University of Nottingham

University Park, Nottingham, NG7 2RD, UK}

{\it E-mail address:} {\bf Mark.Grant@nottingham.ac.uk}
 
\medskip

{\sc Departamento de Matem\'aticas

Universidad de Guanajuato

Guanajuato 36000, Gto, M\'exico}

{\it E-mail address:} {\bf enrique.torres@cimat.mx}
 
\medskip
{\sc Departamento de Matem\'aticas

Centro de Investigaci\'on y de Estudios Avanzados del IPN

M\'exico City 07000, M\'exico}
 
{\it E-mail address}: {\bf xico@math.cinvestav.mx}

\end{document}